\documentclass{article}


\usepackage[backend=biber,giveninits=true,maxbibnames=9,maxcitenames=6]{biblatex}
\bibliography{../refs}

\usepackage{enumerate}

\usepackage[ruled, linesnumbered]{algorithm2e}
\usepackage[table]{xcolor}

\usepackage{booktabs}
\usepackage{longtable}

\SetCommentSty{mycommfont}
\SetKwProg{Fn}{Function}{:}{}

\usepackage{subcaption}

\usepackage{amsmath}
\usepackage{amssymb}
\usepackage{hyperref}

\usepackage{amsthm}
\theoremstyle{definition}
\newtheorem{theorem}{Theorem}[section]
\newtheorem{definition}[theorem]{Definition}
\newtheorem{lemma}[theorem]{Lemma}
\newtheorem{proposition}[theorem]{Proposition}
\newtheorem{corollary}[theorem]{Corollary}
\newtheorem{remark}[theorem]{Remark}
\newtheorem{example}[theorem]{Example}

\DeclareMathOperator{\clamp}{clamp}
\DeclareMathOperator{\Id}{I}
\DeclareMathOperator*{\argmin}{argmin}

\newcommand{\bbR}{\mathbb{R}}

\newcommand{\bbM}{\mathbb{M}}

\newcommand{\bbN}{\mathbb{N}}

\newcommand{\bbRn}{\mathbb{R}^n}

\newcommand{\bbRm}{\mathbb{R}^m}
\newcommand{\bbRnxn}{\mathbb{R}^{n {\times} n}}

\newcommand{\xbar}{\bar{x}}
\newcommand{\ybar}{\bar{y}}
\newcommand{\tbar}{\bar{t}}

\newcommand{\setto}{\rightrightarrows}

\newcommand{\lparan}{\symbol{40}}
\newcommand{\rparan}{\symbol{41}}
\DeclareMathOperator{\dist}{dist}

\usepackage{tikz}
\usepackage{pgfplots}
\pgfplotsset{compat=1.18}
\usetikzlibrary{arrows.meta}
\usetikzlibrary{backgrounds}
\usepgfplotslibrary{patchplots}
\usepgfplotslibrary{fillbetween}
\pgfplotsset{%
    layers/standard/.define layer set={%
        background,axis background,axis grid,axis ticks,axis lines,axis tick labels,pre main,main,axis descriptions,axis foreground%
    }{
        grid style={/pgfplots/on layer=axis grid},%
        tick style={/pgfplots/on layer=axis ticks},%
        axis line style={/pgfplots/on layer=axis lines},%
        label style={/pgfplots/on layer=axis descriptions},%
        legend style={/pgfplots/on layer=axis descriptions},%
        title style={/pgfplots/on layer=axis descriptions},%
        colorbar style={/pgfplots/on layer=axis descriptions},%
        ticklabel style={/pgfplots/on layer=axis tick labels},%
        axis background@ style={/pgfplots/on layer=axis background},%
        3d box foreground style={/pgfplots/on layer=axis foreground},%
    },
}
\usepackage{adjustbox}

\title{A Newton-Kantorovich Inverse Function Theorem in Quasi-Metric Spaces}
\author{Titus Pinta}

\begin{document}
\maketitle

\begin{abstract}
  The purpose of this work is to investigate root finding problems defined on
  (quasi-)metric spaces, and ranging in Euclidean spaces. The motivation for this
  line of inquiry stems from recent models in biology and phylogenetics, where
  problems of great practical significance are cast as optimization problems on
  (quasi-)metric spaces. We investigate a minimal algebraic setup that allows us to
  study a notion of differentiability suitable for Newton-type methods, called
  Newton differentiability. This notion of differentiability benefits from
  calculus rules and is sufficient to prove superlinear convergence of a
  Newton-type method. Finally, a Newton-Kantorovich-type theorem provides an
  inverse function result, applicable on (quasi-)metric spaces.
\end{abstract}



\section{Introduction}
In the last couple of decades, metric spaces have found themselves at the
forefront of optimization and analysis research. A wide range of problems with
important applications in physics, chemistry, and biology require the generality
of metric spaces as a basis for their formulation. In~\cite{LueGatNyeHuc21Wald},
Lueg et al.\ provided a geometrical interpretation to the space of phylogenetic
trees, allowing for the problem of interpolating between two such trees,
constructed on different genetic data, to be formulated as a Fréchet mean problem
on a metric space. The space of protein structures has been given a
\lparan{}pseudo-\rparan{}metric structure in~\cite{RogHen03Anew}, and protein optimal design
problems can be formulated with respect to this structure. This work fits
into the broader topic of knot theory, where metric-geometric aspects have
blended successfully with optimization notions. In the study of our universe,
computational methods for solving shape optimization problems turned out to be necessary. Such
optimization problems for subsets of Finsler manifolds have been solved
in~\cite{ButVel13Shap} and are natural examples of optimization on metric
spaces.

Another key example is provided by problems defined on trees and graphs, as
presented in the recent compendium~\cite{Gol18Opti}.
In this area, ad-hoc combinatorial algorithms and heuristics have long represented
the state of the art. Integer programming is another area with significant
algorithmic development. The recent work on understanding the metric structure
of trees, latices, and graphs can help develop continuous optimization inspired
algorithms. Cyclic projections in Hadamard spaces have been studied
in~\cite{LytPet22Cycl}. Lauster et al.\ developed a fixed point theory for
nonexpansive operators and proximal splitting algorithms in spaces with bounded
curvature in~\cite{BerLauLuk22Alpha} and~\cite{LauLuk21Conv}. This fixed
point approach originates in the work of De Giorgi on minimizing
movements in metric spaces (now called proximity operators in the optimization community), adapted in~\cite{AmbGigSavGrad}.

The interest in quasi-metric spaces, as opposed to just metric spaces, comes
rather naturally from the fact that all the properties employed in the
analysis of nonlinear optimization algorithms are also satisfied by these
more abstract spaces, thus allowing for more general results. The recent
work~\cite{DanSepVen20Asym} extends the construction of free Lipschitz
spaces to quasi-metric spaces, thus showing that such spaces are fruitful
ground for theoretical analysis results.

On the algorithmic side, Definition~\ref{def:newton quasi metric spaces}
from our work develops a Newton-type, superlinearly
convergent algorithm for root finding problems on quasi-metric spaces.
Quasi-metric spaces do not provide a canonical notion of invertible linear maps,
and as such this concept has to be narrowly defined for our purposes.
Riemannian optimization, masterfully presented in the monograph of
Absil~\cite{AbsMahSep08Opti}, provides a blueprint for defining our algebraic
construct, in the Riemannian exponential and logarithm. The main analytical
technique giving rise to superlinear convergence of Newton-type methods
consists of fixed point iterations of operators derived from Newton
differentiable mappings.

In parallel with this algorithmic and fixed point theoretical development, the
basis for an analytical framework in metric spaces has been laid down. A
consensus on the interpretation of gradients in metric spaces has been
established by Hajłasz and Heinonen et al.\ in~\cite{Haj96Sobo,Haj03Sobo,HeiKosSga15Newm}.
The notion of metric slope, introduced by De Giorgi in~\cite{DegMarTos89Prob}, provides an alternative approach to adapting
differential calculus to metric spaces and was presented in~\cite{LauOtt16Conv}
and in the monograph~\cite{AmbGigSavGrad}.
Another area of work consists in exploiting group structures, as seen in
the monograph~\cite{DruKap18Geom}.

The existence of solutions to optimization problems in metric spaces has
attracted the interest of the variational analysis community. The
monograph~\cite{Zas10Opti} collects a large variety of such results. Newton-type
methods have been successfully used in order to obtain inverse function theorems.
The works of Smale~\cite{Sma86Newt} and Kantorovich~\cite{Kan48Func} provide
classic examples of these ideas. The work of Kantorovich has been successfully
applied to nonsmooth problems, defined on Euclidean spaces by Cibulka et
al.\ in~\cite{CibDonPreVelRou18Kant}. Continuing in the realm of nonsmooth
problems defined on linear spaces, Páles used Ekeland's principle for
problems on Banach spaces. A Kantorovich-type result, from
Theorem~\ref{thm:kant}, for equations defined
on quasi-metric spaces stands as the crowning achievement of this chapter.

Calculus rules are vital for algebraic manipulation of the objects involved
in optimization. When dealing with smooth functions, calculus is a well
established discipline, dating to Leibniz and Newton. In recent optimization
work, the calculus of semismooth functions has been collected,
see~\cite{Mov14Nons}. Another approach at nonsmooth calculus is provided by
the work of Bolte and Pauwels, in~\cite{Bol21Cons}. Approaching optimization algorithms
from the point of view of fixed point theory, the paper by Luke, Thao, and
Tham~\cite{LukThaTam18Quan}
provides calculus rules for nonexpansive mappings. The seminal work of
Elliot~\cite{Ell18Thes} shows that such calculus rules are the essence
of automated differentiation, and, in turn, of implementing efficient algorithms.

The first section of this work handles the basic notions and definitions while
helping to fix the notation. The algebraic constructs required by our
Newton-type method follow in the next section. Newton differentiability and
the calculus of Newton differentiability are the subject of the third section,
while the forth section analyses the associated Newton-type method. Forthwith,
the fifth section presents the main result of this work, namely the
Newton-Kantorovich-type inverse function theorem. This result is then followed
by a very simple example application of the Newton type method in the sixth section. The last section of the
article deals with drawing the conclusions.

\subsection{Definitions and Basic Properties}
\begin{definition}
  A space $\bbM$ together with a mapping $\dist:\bbM \times \bbM \to \lbrack 0, \infty \rparan$ is called
  a {\em quasi metric space\/} if the following properties hold:
  \begin{enumerate}[QMS1:]
  \item
    \begin{equation*}
      \forall x, y \in \bbM, x \ne y,\quad \dist(x, y) > 0,
    \end{equation*}
  \item
    \begin{equation*}
      \forall x \in \bbM,\quad \dist(x, x) = 0,
    \end{equation*}
  \item
    \begin{equation*}
      \forall x, y, z \in \bbM,\quad \dist(x, z) \le \dist(x, y) + \dist(y, z),
    \end{equation*}
  \end{enumerate}

  We denote the balls in a quasi-metric space by
  \begin{equation*}
    B_r(x) = \{y \in \bbM~|~\dist(y, x) < r\}
  \end{equation*}
  and
  \begin{equation*}
    B_r[x] = \{y \in \bbM~|~\dist(y, x) \le r\}.
  \end{equation*}
\end{definition}
\begin{remark}
  The difference between metric spaces and quasi-metric spaces lies in the lack
  of symmetry for the distance function. As such, any metric space is also a
  quasi-metric space.
\end{remark}

\begin{definition}[Distance between Points and Subsets]
  Let $\bbM$ be a quasi-metric space and $A \subseteq \bbM$ be a subset. The {\em distance between
    a point $x \in \bbM$ and the subset $A$\/} is defined by
  \begin{equation*}
    \dist(x, A) = \inf_{a \in A} \dist(x, a),
  \end{equation*}
  or
  \begin{equation*}
    \dist(A, x) = \inf_{a \in A} \dist(a, x).
  \end{equation*}
  Let $A, B \subseteq \bbM$. The {\em distance between $A$
    and $B$\/} is defined by
  \begin{equation*}
    \dist(A, B) = \max \{\sup_{a \in A}\dist(a, B), \sup_{b \in B}\dist(A, b)\}.
  \end{equation*}
\end{definition}

For the remaining of this work we will consider the topology generated by
the open balls in $\bbM$. The behavior of such topological spaces has been studied
in~\cite{Kel63Bito}. The topology induces a notion of convergence for sequences.
\begin{definition}
  A sequence ${\{x^k\}}_{k \in \bbN}$ in $\bbM$ is called {\em convergent to $\xbar \in \bbM$\/} if
  \begin{equation*}
    \lim_{k \to \infty} \dist(x^k, \xbar) = 0.
  \end{equation*}
\end{definition}

In order for this work to remain self contained, we recall
the following definition of completeness.
\begin{definition}
  A sequence ${\{x^k\}}_{k \in \bbN}$ in $\bbM$ is called {\em Cauchy\/} if for every
  $\varepsilon > 0$ there is $N \in \bbN$ such that for all $m > N$ and $n > N$,
  $\dist(x^n, x^m) < \varepsilon$.
  A quasi-metric space $(\bbM, \dist)$ is called {\em complete\/} if every Cauchy
  sequence with elements in $\bbM$ is convergent.
\end{definition}

The properties of balls in quasi-metric spaces yield an important topological
property that will be needed in the proof of the Banach Fixed Point Theorem.
\begin{lemma}\label{lema:quasi-metric intersecrtion of infinite balls}
  Let $\xbar \in \bbM$ and ${\{r_k\}}_{k \in \bbN}$ be a sequence of real numbers. If
  $\lim_{k \to \infty}r_k = 0$, then
  \begin{equation*}
    \bigcap_{k \in \bbN}B_{r_k}[\xbar] = \{\xbar\}.
  \end{equation*}
\end{lemma}
\begin{proof}
  Clearly, $\xbar \in B_{r_k}[\xbar]$ for any $k \in \bbN$, so
  \begin{equation*}
    \bigcap_{k \in \bbN}B_{r_k}[\xbar] \supseteq \{\xbar\}.
  \end{equation*}
  Next, assume $y \in \bigcap_{k \in \bbN}B_{r_k}[\xbar]$ with $y \ne \xbar$, so there exists $\varepsilon > 0$
  with $\dist(y, \xbar) > \varepsilon$ and $N \in \bbN$ with $r_k < \varepsilon$ for all $k \ge N$. This
  shows that $y \not\in B_{r_k}[\xbar]$ for any $k \ge N$, thus contradicting
  the assumption. It follows that the considered intersection is a singleton,
  completing the proof.
\end{proof}

With quasi-metric spaces as a topological background, we can define the fixed
point iterations of set-valued operators and describe their behavior.
\begin{definition}
  Let $T:U \subseteq \bbM \setto \bbM$, then a point $\xbar \in U$ is called a {\em fixed point\/} if
  $\xbar \in T(\xbar)$.
\end{definition}

The speed of the convergence to a limit point can be quantified, based on the
following definitions.
\begin{definition}[Convergence Rates]
  Let $(\bbM, \dist)$ be a quasi-metric space.
  A convergent sequence ${\{x^k\}}_{k \in \bbN} \subseteq \bbM$ with $\lim_{k \to \infty}x^k = \xbar$ is
  called {\em linearly convergent\/} if there exists $c < 1$ such that
  \begin{equation*}
    \forall k \in \bbN,\quad \dist(x^{k+1}, \xbar) \le c\dist(x^k, \xbar).
  \end{equation*}
  A convergent sequence ${\{x^k\}}_{k \in \bbN}$ with $\lim_{k \to \infty}x^k = \xbar$ is called
  {\em superlinearly convergent\/} if there exists a sequence
  ${\{c^k\}}_{k \in \bbN}$ such that
  \begin{equation*}
    \forall k \in \bbN,\quad \dist(x^{k+1}, \xbar) \le c^k \dist(x^k, \xbar)
  \end{equation*}
  with $\lim_{k \to \infty} c^k = 0$.
  A convergent sequence ${\{x^k\}}_{k \in \bbN}$ with $\lim_{k \to \infty}x^k = \xbar$ is called
  {\em convergent with rate $\gamma > 1$\/} if there exists $c < 1$ such that
  \begin{equation*}
    \forall k \in \bbN,\quad \dist(x^{k+1}, \xbar) \le c {\dist(x^k, \xbar)}^\gamma.
  \end{equation*}
  A convergent sequence with rate 2 is called {\em quadratically convergent}.
  A convergent sequence ${x^k}_{k \in \bbN}$ with $\lim_{k \to \infty}x^k = \xbar$ is called
  {\em convergent with super-rate $\gamma > 1$\/} if there exists a sequence
  ${\{c^k\}}_{k \in \bbN}$ such that
  \begin{equation*}
    \forall k \in \bbN,\quad \dist(x^{k+1}, \xbar) \le c^k {\dist(x^k, \xbar)}^\gamma
  \end{equation*}
  with $\lim_{k \to \infty} c^k = 0$.
\end{definition}

In order to develop a system capable of analyzing the behavior of fixed point
operators, we need to recall multiple definitions relating to the smoothness
properties of operators.
\begin{definition}
  A set value mapping $T: U \subset \bbM \setto \bbM$ is called {\em proper\/} on $U$ if
  $\nexists x \in U$ such that $F(x) = \emptyset$.
\end{definition}

\begin{definition}
  A mapping $T:U \subseteq \bbM \setto \bbM$ is called
  \begin{enumerate}
  \item {\em Lipschitz continuous on $U$\/}
    if there exists a constant $L \ge 0$ such that
    \begin{equation*}
      \forall x, \xbar \in U, y \in T(x), \ybar \in T(\xbar),\quad \dist(y, \ybar) \le L\dist(x, \xbar),
    \end{equation*}

  \item {\em contraction on $U$\/} if
    there exists a constant $c < 1$ such that
    \begin{equation*}
      \forall x, \xbar \in U, y \in T(x), \ybar \in T(\xbar),\quad \dist(y, \ybar) \le c\dist(x, \xbar),
    \end{equation*}

  \item {\em Hölder continuous on $U$\/} if
    there exist constants $L \ge 0$ and $\alpha > 0$ such that
    \begin{equation*}
      \forall x, \xbar \in U, y \in T(x), \ybar \in T(\xbar),\quad \dist(y, \ybar) \le L{\dist(x, \xbar)}^\alpha.
    \end{equation*}
  \end{enumerate}
\end{definition}

These definitions have weaker, pointwise analogues that provide an appropriate
environment for the convergence analysis presented in most of this work.
\begin{definition}
  A mapping $T:U \subseteq \bbM \setto \bbM$ is called
  \begin{enumerate}
  \item {\em pointwise Lipschitz continuous at $\xbar$\/}
    if there exists a constant $L \ge 0$ such that
    \begin{equation*}
      \forall x \in U, y \in T(x), \ybar \in T(\xbar),\quad \dist(y, \ybar) \le L\dist(x, \xbar),
    \end{equation*}

  \item {\em quasi-contraction at $\xbar$\/} if
    there exists a constant $c < 1$ such that
    \begin{equation*}
      \forall x \in U, y \in T(x), \ybar \in T(\xbar),\quad \dist(y, \ybar) \le c\dist(x, \xbar),
    \end{equation*}

  \item {\em pointwise Hölder continuous on $\xbar$\/} if
    there exist constants $L \ge 0$ and $\alpha > 0$ such that
    \begin{equation*}
      \forall x \in U, y \in T(x), \ybar \in T(\xbar),\quad \dist(y, \ybar) \le L{\dist(x, \xbar)}^\alpha.
    \end{equation*}
  \end{enumerate}
\end{definition}

\begin{remark}
  A mapping is a {\lparan}quasi-{\rparan}contraction at $\xbar$ if and only if it is
  (pointwise) Lipschitz continuous at $\xbar$ with constant $L < 1$.
\end{remark}

In the study of set-valued mappings it suffices to consider the existence of
a smooth selection in order to characterize the behavior of fixed point
iterations. For this purpose we provide the next definitions.
\begin{definition}\label{d:holder-selection}
  A set-valued mapping $F:U \subseteq \bbM \setto \mathcal{M}$ has a
  \begin{enumerate}
  \item {\em (pointwise) Lipschitz selection}, if there exists a (pointwise)
    Lipschitz mapping $f:U \to \mathcal{M}$ with $f(x) \in F(x)$ for all $x \in U$,
  \item {\em (pointwise) Hölder selection}, if there
    exists a (pointwise) Hölder mapping $f:U \to \bbRm$ with $f(x) \in F(x)$ for all
    $x \in U$,
  \item {\em {\lparan}quasi-{\rparan}contractive selection}, if there exists a
    {\lparan}quasi-{\rparan}contractive mapping $f:U \to \mathcal{M}$ with $f(x) \in F(x)$ for all
    $x \in U$.
  \end{enumerate}
\end{definition}

\begin{lemma}
  Let $T: U \subseteq \bbM \setto \bbM$ be pointwise Lipschitz at $\xbar \in \bbM$, then it is single-valued
  at $\xbar$.
\end{lemma}
\begin{proof}
  Let $\ybar_1, \ybar_2 \in T(\xbar)$, so
  \begin{equation*}
    \dist(\ybar_1, \ybar_2) \le L\dist(\xbar, \xbar) \le 0.
  \end{equation*}
\end{proof}
\begin{example}
  The mapping $F:\bbR ^2 \setto \bbR^2$ defined by
  \begin{equation*}
    F(x,y) = \{(x, 0), (y, 0)\}
  \end{equation*}
  is pointwise Lipschitz continuous at $(\xbar, \ybar) = (0, 0)$, as can be seen from
  \begin{equation*}
    \|(x, 0) - (0, 0)\| = x \le \sqrt{x^2 + y^2} = \|(x, y) - (0, 0)\|
  \end{equation*}
  and
  \begin{equation*}
    \|(y, 0) - (0, 0)\| = y \le \sqrt{x^2 + y^2} = \|(x, y) - (0, 0)\|.
  \end{equation*}
  We can see that at $(\xbar, \ybar) = (0, 0)$ the mapping is single-valued. It is not
  single-valued at any other point and as such it is only pointwise Lipschitz
  continuous at $(\xbar, \ybar) = (0, 0)$.
\end{example}

The fundamental result of the Banach Contraction Mapping Principle has been
successfully extended to set-valued mappings in~\cite{Nad69Mult}, while the
proof in the setting of quasi-metric spaces has been developed
in~\cite{SecMatWar19Newf}. We provide here a version of this proof adapted to
work for set-valued quasi-contractions on quasi-metric spaces.
\begin{theorem}[Banach Fixed Point]\label{t:BFP}
  Let $\bbM$ be a quasi-metric space and $T: V \subseteq \bbM \setto \bbM$ be a quasi-contraction
  at $\xbar \in \bbM$ with $T(\xbar) = \xbar$, then there exists a neighborhood $O$ of $\xbar$ with
  $U := O\cap V$ such that $T:U \setto U$, and any sequence
  ${\{x^k\}}_{k \in \bbN}$ with $x^0 \in U$ and $x^{k+1} \in T(x^k)$ converges at least
  linearly to the unique fixed point $\xbar$.
\end{theorem}

\begin{proof}
  First, we will show that $\xbar$ is the unique fixed point. Consider $\ybar \in \bbM$ a
  fixed point, i.e. $T(\ybar) = \{\ybar\}$, and from quasi-contractivity
  \begin{equation*}
    \dist(\ybar, \xbar) \le c \dist(\ybar, \xbar).
  \end{equation*}
  Since $c \in (0, 1)$, we can conclude that $\dist(\ybar, \xbar) = 0$ and thus $\xbar = \ybar$.

  The mapping $F$ sends balls around $\xbar$ into themselves because for $x \in \bbM$ with
  $\dist(x, \xbar) < r$, quasi-contractivity implies $\sup_{y \in F(x)}\dist(y, \xbar) < cr$.
  Let ${\{x^k\}}_{k \in \mathbb{N}}$ be a sequence generated by $x^{k+1} \in F(x^k)$ with
  $x^0 \in U$ and let $r = \dist(x^0, \xbar)$. We can conclude that
  \begin{equation*}
    x_k \in B_{{c^k}r}(\xbar),
  \end{equation*}
  so $\dist(x^k, \xbar) \le c^k r$, and the sequence ${\{x^k\}}_{k \in \bbN}$ is convergent, and
  \begin{equation*}
    \lim_{k \to \infty} x^k \in \bigcap_{k \in \bbN} B_{{c^k}r}(\xbar) = \{\xbar\},
  \end{equation*}
  where we used Lemma~\ref{lema:quasi-metric intersecrtion of infinite balls}
  to compute the intersection of the balls.
\end{proof}

\begin{remark}
  When applied to a quasi-contraction, as opposed to a contraction, the
  Banach fixed point Theorem cannot guarantee the existence of a fixed point
  and because of this, Theorem~\ref{t:BFP} requires the existence of a fixed
  point as an assumption. This allows a weakening of the topological
  assumptions, namely that we do not require completeness of the space.
\end{remark}

\section{Algebraic Constructs}
The invertible linear mappings that are normally used in the development of
Newton's method turn out to be more rigid than actually required for
obtaining superlinear convergence rates. Two insights provide the suitable
definition for pseudo-linear maps on quasi-metric spaces.  The first insight is
that the only linearity property that is actually employed is the fact that
linear transformations map 0 to 0.  The second insight is that we can identify
two points with their difference vector, and as such we can consider linear
mappings equivalently as acting on pairs, or differences, of points.  These
observations lead to the following definition of linear mappings on quasi-metric
spaces. For the remaining of this chapter, $(\bbM, \dist)$ will be a quasi-metric
space.
\begin{definition}[Pseudo-Linear Maps]
  A mapping $H:\bbM \times \bbM \to \bbRn$ is called \emph{pseudo-linear} if $\forall x \in \bbM$,
  $H(x, x) = 0$. We denote the space of such mappings with $S_n(\bbM)$.
\end{definition}

Likewise, the full range of properties that are associated with inverses of
linear mappings in linear spaces is not required.  For Newton-type methods, all
we require is the following notion.
\begin{definition}[Inversely Compatible Maps]\label{def:inversley compatible}
  A pseudo-linear mapping $H:\bbM \times \bbM \to \bbRn$ is called \emph{inversely compatible}
  if there exists $H^{-}:\bbM \times \bbRn \to \bbM$ and $m \in \bbR$ such that
  \begin{equation*}
    \forall x, \in \bbM,\quad H^{-}(x, 0) = x,
  \end{equation*}
  and
  \begin{equation}\label{eq:def-geodesical-subadditive}
    \forall x, y \in \bbM, \forall v, w \in \bbRn,\quad \dist(H^{-}(x, v), H^{-}(y, w)) \le m\|v - w - H(x, y)\|.
  \end{equation}
  The mapping $H^{-}$ is called a {\em quasi inverse\/} of $H$. We denote
  \begin{equation}\label{eq:def-geodesical-operator norm}
    \||H^{-}\||:= \inf\{m \in \bbR~|~\mbox{\eqref{eq:def-geodesical-subadditive} holds} \}.
  \end{equation}
  The set of all such mappings is denoted by $GS_n(\bbM)$
\end{definition}

These notions simplify back to known objects in the case of Euclidean spaces,
by interpreting pseudo-linear mappings as linear maps acting on the difference
vector of two points. Similarly, on Riemannian Manifolds, pseudo-linear maps
act on the vector produced by the Riemannian logarithm of two points. The
next example clarifies both how these notions map back to Euclidean spaces and
where the inspiration for them comes from.
\begin{example}\label{ex:metric-space-euclidean}
  Let $\bbM = \bbRn$ with the Euclidean metric and let $T \in \bbRnxn$. Then the mapping
  $H:\bbRn \times \bbRn \to \bbRn$ defined by
  \begin{equation*}
    H(x, y) = T(y - x)
  \end{equation*}
  is pseudo-linear, because clearly
  \begin{equation*}
    H(x, x) = T(x - x) = 0.
  \end{equation*}
  Furthermore, if $T$ is invertible, then $H$ is inversely compatible, with
  \begin{equation*}
    H^{-}(x, v) = x + T^{-1}v.
  \end{equation*}
  This can be seen because
  \begin{align}\label{eq:example-linear-is-skew-symetric}
    \forall x \in \bbRn, y \in \bbRn, v \in \bbRn, w \in \bbRn \dist(H^{-}(x, v), H^{-}(y, w))
    &= \|H^{-}(y, w) - H^{-}(x, v)\| \nonumber \\
    &= \|y + T^{-1}w - x - T^{-1}v\| \nonumber \\
    &\le \|T^{-1}\|\|T(y - x) + w - v\| \nonumber \\
    &= \|T^{-1}\|\|v - w -T(y - x)\|.
  \end{align}
\end{example}
\begin{remark}
  The computation in~\eqref{eq:example-linear-is-skew-symetric} justifies the
  notation $\||H^{-}\||$ from~\eqref{eq:def-geodesical-operator norm} because in
  the Euclidean metric, considering the pseudo-linear mapping induced by
  a linear mapping $T$, $\||H^{-}\|| = \|T^{-1}\|$ holds.
\end{remark}

It is clear that the space $S_n(\bbM)$ inherits the algebraic structure of $\bbRn$,
so for any $H_1, H_2 \in S_n(\bbM)$, there exists $H_1 + H_2 \in S_n(\bbM)$ and
$\langle H_1, H_2 \rangle \in S_1(\bbM)$ defined by $(H_1 + H_2)(x, y) = H_1(x, y) + H_2(x, y)$
and $\langle H_1, H_2 \rangle(x, y) = \langle H_1(x, y), H_2(x, y) \rangle$ respectively. Similarly,
for $H_3, H_4 \in S_1(\bbM)$, there exist $H_3\cdot H_4 \in S_1(\bbM)$ and $H_3 \cdot H_1 \in S_n(\bbM)$
defined by $(H_3\cdot H_4)(x, y) = H_3(x, y)H_4(x, y)$ and
$(H_3\cdot H_1)(x, y) = H_3(x, y)H_1(x, y)$ respectively. Finally, the mapping
$H_1 \oplus H_2 \in S_{2n}(\bbM)$ is defined by $(H_1\oplus H_2)(x, y) = H_1(x, y) \oplus H_2(x, y)$.

\section{Newton Differentiability}
The work presented in~\cite{Qi_Sun93Anon} has introduced Newton's method for
semismooth functions. The defining property that yields superlinear
convergence can be adapted to functions defined on quasi-metric spaces, using
the previously developed algebraic notions.
\begin{definition}[Pointwise Newton Differentiability in Quasi-Metric Spaces]%
\label{d:Newton Diff metric}
  Let $\bbM$ be a quasi-metric space. A function $F:\bbM \to \bbRn$ is called
  \emph{weakly pointwise Newton differentiable at $\xbar$} if there exists a set
  valued mapping $\mathcal{H}F:\bbM \setto S_n(\bbM)$ such that
  \begin{equation}\label{eq:def-newton-diff-weak-metric}
    \lim_{x \to \xbar}\sup_{H \in \mathcal{H}F(x)}\frac{\|F(x) - F(\xbar) - H(x, \xbar)\|}{\dist(x, \xbar)} < \infty.
  \end{equation}
  Furthermore, if
  \begin{equation}\label{eq:def-newton-diff-metric}
    \lim_{x \to \xbar}\sup_{H \in \mathcal{H}F(x)}\frac{\|F(x) - F(\xbar) - H(x, \xbar)\|}{\dist(x, \xbar)} = 0,
  \end{equation}
  the function $F$ is called \emph{pointwise Newton differentiable at $\xbar$}.
\end{definition}

When studying Newton differentiability at all the points in a set, we can look
at a stronger smoothness condition, namely that of uniform Newton
differentiability.
\begin{definition}[Uniform Newton differentiability in Quasi-Metric Spaces]%
\label{d:uniform Newton Diff metric}
  Let $\bbM$ be a quasi-metric space. A function $F:\bbM \to \bbRn$ is called
  \emph{weakly uniformly Newton differentiable on $V \subseteq \bbM$} if there exists a set
  valued mapping $\mathcal{H}F:\bbM \setto S_n(\bbM)$ such that for every $\varepsilon > 0$ there
  exists a $\delta > 0$ such that for all $x \in \bbM$ and all $y \in V$ with $\dist(x, y) \le \delta$,
  \begin{equation}\label{eq:uniform def-newton-diff-weak-metric}
    \sup_{H \in \mathcal{H}F(x)}\frac{\|F(x) - F(y) - H(x, y)\|}{\dist(x, y)}
    \in (c + \varepsilon, c - \varepsilon).
  \end{equation}

  Furthermore, when for every $\varepsilon > 0$ there
  exists a $\delta > 0$ such that for all $x \in \bbM$ and all $y \in V$ with $\dist(x, y) \le \delta$,
  \begin{equation}\label{eq:uniform def-newton-diff--metric}
    \sup_{H \in \mathcal{H}F(x)}\frac{\|F(x) - F(y) - H(x, y)\|}{\dist(x, y)} < \varepsilon,
  \end{equation}
  the function $F$ is called \emph{uniformly Newton differentiable at $\xbar$}.
\end{definition}
\begin{remark}
  The equations~\eqref{eq:uniform def-newton-diff-weak-metric}
  and~\eqref{eq:uniform def-newton-diff--metric} imply the
  convergence of the limits in equations~\eqref{eq:def-newton-diff-weak-metric}
  and~\eqref{eq:def-newton-diff-metric} respectively.
  Even more, the convergence is uniform in $y$.
\end{remark}
\begin{remark}[Subsets of Newton Differential]\label{remark:subsets}
  It is useful to remark that subsets of a Newton differential are still
  Newton differentials. This fact follows clearly by remarking that
  if the supremum is taken over a smaller set, its value cannot increase.
  As such, the Newton differential of a function is not a unique object.
\end{remark}

In order to better understand these notions, we can consider the case of
$\bbRn$ equipped  with the standard Euclidean metric. In this setting, all
sufficiently smooth functions are Newton differentiable, with Newton
differentials constructed from the traditional Fréchet differentials.
\begin{proposition}[Fréchet Differentiability and Newton Differentiability]
  \label{prop:example-frechet}
  Let $F: \bbRm \to \bbRn$ of class $\mathcal{C}^1$. Then $F$ is pointwise Newton differentiable
  at any $\xbar \in \bbRm$ with Newton differential $\mathcal{H}F(x) = \{(y, z) -> \nabla F(x)^T(z - y)\}$.
\end{proposition}
\begin{proof}
  Using Taylor's expansion for $F$ around $\xbar$, we see that there exists
  $h:\bbRm \to \bbRn$ such that $\lim_{x \to \xbar}h(x) = 0$ and
  \begin{equation*}
    F(x) = F(\xbar) + (\nabla F(\xbar)^T - \nabla F(x)^T + \nabla F(x)^T)(x - \xbar) + h(x)\|x - \xbar\|.
  \end{equation*}
  Clearly rearranging, dividing by $\|x - \xbar\|$ and taking the limit as $x \to \xbar$
  while using the continuity of $\nabla F$ at $\xbar$ shows the desired conclusion.
\end{proof}

\subsection{Calculus of Newton Differentiability}
As Newton differentiability is defined in a manner analogues to Fréchet
differentiability, similar techniques can be used in order to construct
calculus rules for Newton differentiable functions.
\begin{proposition}\label{prop:calculus-newton-diff-sum}
  Let $F:U \subseteq \bbM \to \bbRn$ and $G:U \to \bbRn$ be pointwise Newton differentiable at $\xbar \in U$
  with Newton differentials $\mathcal{H}F$ and $\mathcal{H}G$. Then $F + G$ is Newton
  differentiable at $\xbar$ with Newton differential
  $x \mapsto \{H_F + H_G ~|~ H_F \in \mathcal{H}F(x), H_G \in \mathcal{H}G(x)\}$.
\end{proposition}
\begin{proof}
  From the assumptions, we know that
  \begin{equation}\label{eq:calculus-of-newton--diff-eq-1}
    \lim_{x \to \xbar}\sup_{H_F \in \mathcal{H}F(x)}\frac{\|F(x) - F(\xbar) - H_F(x, \xbar)\|}{\dist(x, \xbar)} = 0
  \end{equation}
  and
  \begin{equation}\label{eq:calculus-of-newton--diff-eq-2}
    \lim_{x \to \xbar}\sup_{H_G \in \mathcal{H}G(x)}\frac{\|G(x) - G(\xbar) - H_G(x,  \xbar)\|}{\dist(x, \xbar)} = 0.
  \end{equation}

  The triangle inequality together with the sub-additivity of the supremum
  operator provides the next step in the proof, yielding
  \begin{align*}
   &\sup_{H_F + H_G \in \mathcal{H}F(x) + \mathcal{H}G(x)}
    \frac{\|(F+G)(x) - (F+G)(\xbar) - (H_F + H_G)(x, \xbar)\|}{\dist(x, \xbar)} \nonumber \\
   &\le \sup_{H_F + H_G \in \mathcal{H}F(x) + \mathcal{H}G(x)}
     \frac{\|F(x) - F(\xbar) - H_F(x, \xbar)\|}{\dist(x, \xbar)}
     + \frac{\|G(x) - G(\xbar) - H_G(x, \xbar)\|}{\dist(x, \xbar)} \nonumber \\
   &\le \sup_{H_F \in \mathcal{H}F(x)} \frac{\|F(x) - F(\xbar) - H_F(x, \xbar)\|}{\dist(x, \xbar)}
     + \sup_{H_G \in \mathcal{H}G(x)} \frac{\|G(x) - G(\xbar) - H_G(x, \xbar)\|}{\dist(x, \xbar)},
  \end{align*}
  and taking the limit as $x \to \xbar$ using~\eqref{eq:calculus-of-newton--diff-eq-1}
  and~\eqref{eq:calculus-of-newton--diff-eq-2} finishes the proof.
\end{proof}

For the chain rule, we consider $\bbRn$ equipped with the standard Euclidean
distance.
\begin{proposition}\label{prop:calculus-newton-diff-compositon}
  Let $F:U \subseteq \bbM \to \bbRm$ and $G:F(U) \to \bbRn$ be pointwise Newton differentiable at $\xbar$
  and $F(\xbar)$ respectively with Newton differentials $\mathcal{H}F$ and $\mathcal{H}G$. Assume
  further that $F$ is continuous at $\xbar$ and that there exists $K > 0$ such that
  \begin{equation}\label{eq:chain rule bound on H_F}
    \sup_{x \in U}\sup_{H \in \mathcal{H}F(x)}\sup_{y, z \in U}\|H(x)(y, z)\| \le K\dist(y, z).
  \end{equation}
  Then $G \circ F$ is Newton differentiable at $\xbar$ with Newton differential
  \begin{equation*}
    \mathcal{H}(G \circ F)(x) = \{(y, z) \mapsto H_G(F(y), F(z))~|~H_G \in \mathcal{H}G(F(x))\}.
  \end{equation*}
\end{proposition}
\begin{proof}
  We first need to establish a bound on
  \begin{equation*}
    \lim_{x \to \xbar}\frac{\|F(x) - F(\xbar)\|}{\dist(x, \xbar)}.
  \end{equation*}
  From the Newton differentiability of $F$ we can conclude that there exists a
  neighborhood $V$ of $\xbar$ such that for all $x \in V$, $x \ne \xbar$
  \begin{equation}\label{eq:in proof of chain rule 1}
        \sup_{H \in \mathcal{H}F(x)}\frac{\|F(x) - F(\xbar) - H(x, \xbar)\|}{\dist(x, \xbar)} \le 1.
  \end{equation}
  With~\eqref{eq:in proof of chain rule 1} and~\eqref{eq:chain rule bound on H_F},
  we can use the triangle inequality for the norm to bound
  \begin{align}\label{eq:chain rule proof f}
    \|F(x) - F(\xbar)\|
    &\le \sup_{H \in \mathcal{H}F(x)}\|F(x) - F(\xbar) - H(x, \xbar) + H(x, \xbar)\| \nonumber \\
    &\le \sup_{H \in \mathcal{H}F(x)}\|F(x) - F(\xbar) - H(x, \xbar)\| + \|H(x, \xbar)\| \nonumber \\
    &\le \dist(x, \xbar) + K\dist(x, \xbar).
  \end{align}

  We can now focus our attention on the key object for the proof at hand and
  compute, using~\eqref{eq:chain rule proof f},
  \begin{align}\label{eq:chain rule proof thing 2}
    \lim_{x \to \xbar}
    &\sup_{H \in \mathcal{H}(F \circ G)(x)} \frac{\|G \circ F(x) - G \circ F(\xbar) - H(x, \xbar)\|}{\dist(x, \xbar)}
      \nonumber \\
    &\le\lim_{x \to \xbar}\sup_{H \in \mathcal{H}(F \circ G)(x)}
      \frac{\|G \circ F(x) - G \circ F(\xbar) - H(x, \xbar)\|}{\|F(x) - F(\xbar)\|}
      \frac{\|F(x) - F(\xbar)\|}{\dist(x, \xbar)}\nonumber \\
    &\le\lim_{x \to \xbar}\sup_{H \in \mathcal{H}(F \circ G)(x)}
      (K+1)\frac{\|G \circ F(x) - G \circ F(\xbar) - H(x, \xbar)\|}{\|F(x) - F(\xbar)\|}.
  \end{align}
  The last step consists in using the continuity of $F$ at $\xbar$ and the
  Newton differentiability of $G$ at $F(\xbar)$ to calculate
  \begin{equation}\label{eq:chain rule proof thing 3}
    \lim_{x \to \xbar}\sup_{H \in \mathcal{H}G(F(x))}
      \frac{\|G \circ F(x) - G \circ F(\xbar) - H_G(F(x), F(\xbar))\|}{\|F(x) - F(\xbar)\|} = 0.
  \end{equation}
  Combining~\eqref{eq:chain rule proof thing 3}
  with~\eqref{eq:chain rule proof thing 2} proves the conclusion.
\end{proof}
\begin{remark}
  It is interesting to see how this chain rule behaves in the context of
  Example~\ref{ex:metric-space-euclidean}. For this, consider $F, G: \bbRn \to \bbRn$
  of class $\mathcal{C}^{\infty}$. From proposition~\ref{prop:example-frechet}, the Newton
  differential of $G$ at $F(x)$ is $(y, z) \mapsto \nabla G(F(x))(z - y)$. Assuming
  that $\|\nabla F(x)\|$ is bounded on a neighborhood of $\xbar$ we can compute the
  Newton differential of $\mathcal{H}(G \circ F)$ as the singleton
  $\mathcal{H}(G \circ F)(x) = \{(y, z) \mapsto \nabla G(F(x))^T(F(z) - F(y))\}$. Looking at the key
  object in~\eqref{eq:def-newton-diff-metric} and using the mean value
  theorem allows us to relate the chain rule for Fréchet differentiability
  with that of Newton differentiability by computing
  \begin{equation*}
    \|\mathcal{H}(G \circ F)(x)(x, \xbar)\| = \|\nabla G(F(x))^T(F(\xbar) - F(x))\| \le
    \|\nabla G(F(x))^T\|\nabla F(\xi)^T\|\|\xbar - x\|,
  \end{equation*}
  where $\xi$ is a point in the line segment between $x$ and $\xbar$.
\end{remark}

In order to complete the fundamental calculus rules, we need to describe
the behavior of the direct sum of two Newton differentiable functions.
\begin{proposition}
  Let $F:U \subseteq \bbM \to \bbRn$ and $G:U \to \bbRm$ be pointwise Newton differentiable at $\xbar$
  with Newton differentials $\mathcal{H}F$ and $\mathcal{H}G$. Then $F \oplus G:U \to \bbR^{n + m}$ is Newton
  differentiable at $\xbar$ with Newton differential
  $x \mapsto \{H_F \oplus H_G~|~H_F \in \mathcal{H}F(x), H_G \in \mathcal{H}G(x)\}$.
\end{proposition}
\begin{proof}
  The proof simply follows from the fact that for any $x \in \bbRn$ and $y \in \bbRm$,
  \begin{equation*}
    \|x \oplus y\|^2 = \|x\|^2 + \|y\|^2.
  \end{equation*}
  Applying this to the defining property of Newton differentiability,
  \begin{align*}
    \lim_{x \to \xbar}
    &\sup_{(H_F \oplus H_G) \in \mathcal{H}(F \oplus G)(x)}
      \frac{\|(F \oplus G)(x) - (F \oplus G)(\xbar) - (H_F \oplus H_G)(x, \xbar)\|^2}{{\dist(x, \xbar)}^2}\\
    &= \lim_{x \to \xbar}\sup_{H_F \in \mathcal{H}F(x)}
      \frac{\|F(x) - F(\xbar) - H_F(x, \xbar)\|^2}{{\dist(x, \xbar)}^2} \\
    &+ \lim_{x \to \xbar}\sup_{H_G \in \mathcal{H}G(x)}
      \frac{\|G(x) - G(\xbar) - H_G(x, \xbar)\|^2}{{\dist(x, \xbar)}^2} \\
    &= 0,
  \end{align*}
  where we have used the Newton differentiability of $F$ and $G$ to compute the
  two limits.
\end{proof}

As the last step necessary in order to provide complete calculus rules,
we will prove the product rule for Newton differentiable functions.
\begin{proposition}
  Let $F:U \subseteq \bbM \to \bbR$ and $G:U \to \bbR$ be pointwise Newton differentiable at $\xbar$
  with Newton differentials $\mathcal{H}_F$ and $\mathcal{H}_G$. Then $F \cdot G$ is Newton
  differentiable at $\xbar$ with Newton differential
  \begin{equation*}
    \mathcal{H}(F \cdot G)(x) = \{(y, z) \mapsto H_F(y, z)G(x) + F(x)H_G(y, z)~|~H_G \in \mathcal{H}G(x), H_F \in \mathcal{H}F(x)\}.
  \end{equation*}
\end{proposition}
\begin{proof}
  Using a similar argument as in Proposition~\ref{prop:calculus-newton-diff-sum}, we can
  directly compute
  \begin{align*}
   \lim_{x \to \xbar}&\sup_{H \in \mathcal{H}(F \cdot G)(x)}
    \frac{\|(F \cdot G)(x) - (F \cdot G)(\xbar) - H(x, \xbar)\|}{\dist(x, \xbar)} \\
   &=\lim_{x \to \xbar}\sup_{H_F \in \mathcal{H}F(x), H_G \in \mathcal{H}G(x)}
     \frac{\|F(x)G(x) - F(\xbar)G(\xbar) - H_F(x, \xbar)G(x) - F(x)H_G(x, \xbar)\|}{\dist(x, \xbar)} \\
   &=\lim_{x \to \xbar}\sup_{H_F \in \mathcal{H}F(x), H_G \in \mathcal{H}G(x)}
     \frac{\|F(x)(G(x) - G(\xbar) - H_G(x, \xbar))-
     (F(x) - F(\xbar) - H_F(x, \xbar)G(x))G(\xbar)\|}{\dist(x, \xbar)} \\
   &\le\lim_{x \to \xbar}\sup_{H_G \in \mathcal{H}G(x)} \frac{\|F(x)(G(x) - G(\xbar) - H_G(x, \xbar))\|}{\dist(x, \xbar)}
     + \sup_{H_F \in \mathcal{H}F(x)} \frac{\|(F(x) - F(\xbar) - H_F(x, \xbar))G(\xbar)\|}{\dist(x, \xbar)}\\
    &= 0,
  \end{align*}
  where the last two limits are zero because of the Newton differentiability of
  $F$ and $G$.
\end{proof}

\section{Newton-type Methods}
The class of Newton differentiable functions provides a large pool of candidates
for Newton-type methods. Under the assumption that a Newton differential of
the function $F$ contains \emph{inversely compatible} pseudo-linear maps,
a Newton-type fixed point operator can be defined.
As explained in Example~\ref{ex:metric-space-euclidean}, the inverse of
a pseudo-linear map acts as a translation on points in $\bbM$. This motivates
the adaptation of the classic Newton's method in our setting.
\begin{definition}[Newton-type Method in Quasi-metric Spaces]%
  \label{def:newton quasi metric spaces}
  The fixed point iteration of the proper (nowhere empty) set-valued operator
  $\mathcal{N}_{\mathcal{H}F}:M \setto M$,
  \begin{equation*} 
    \mathcal{N}_{\mathcal{H}F} x = \{{H}^{-}(x, -F(x))~|~H \in \mathcal{H}F(x) \cap GS_n(\bbM)\},
  \end{equation*}
  \begin{equation}\label{eq:Metric-Newton-type-method}
    x^{k+1} \in \mathcal{N}_{\mathcal{H}F}{x}^k
  \end{equation}
  is called a {\em Newton-type method}.
\end{definition}

In order to analyze the convergence rate of this method, we first need to
establish convergence to a fixed point. For this, we employ the Banach Fixed
Point Theorem.
\begin{proposition}\label{prop:weak-metric-newton-is-quasicontraction}
  Let $F: U \subseteq \bbM \to \bbRn$ be pointwise weakly Newton differentiable at $\xbar$
  with $F(\xbar) = 0$.
  Denote the Newton differential of $F$ by $\mathcal{H}F$, and assume that $\forall x \in U$ all
  $H \in \mathcal{H}(x)$ are inversely compatible mappings, that is $\mathcal{H}(x) \subseteq GS_n(\bbM)$.
  Furthermore, assume that the set $\bigcup_{x \in U}\{\||{{H}^{-}}\||~|~H \in \mathcal{H}F(x)\}$
  is bounded by $\Omega > 0$ and let $c > 0$ be the limit
  in~\eqref{eq:def-newton-diff-weak-metric}, i.e.
  \begin{equation*}
    \lim_{x \to \xbar}\sup_{H \in \mathcal{H}F(x)}\frac{\|F(x) - F(\xbar) - H(x, \xbar)\|}{\dist(x, \xbar)} = c,
  \end{equation*}
  be such that $c \Omega < 1$. Then there exists a neighborhood $\xbar \in V \subseteq \bbM$ such
  that the mapping $\mathcal{N}_{\mathcal{H}F}$ is a quasi-contraction on $V$.
\end{proposition}
\begin{proof}
  From the Archimedean principle, there exists $\varepsilon > 0$ such that
  $(c + \varepsilon) \Omega < 1$. Based on~\eqref{eq:def-newton-diff-weak-metric}, there
  exists a neighborhood $\xbar \in V \subseteq U$ such that
  \begin{equation}\label{eq:newton-type-method-is-a-contraction-metric}
    \forall x \in V,\quad \sup_{H \in \mathcal{H}F(x)} \|F(x) - F(\xbar) - H(x)(x, \xbar)\| \le (c + \varepsilon)\dist(x, \xbar).
  \end{equation}
  Let $y \in \mathcal{N}_{\mathcal{H}F}(x)$ such that $y = H^{-}(x, -F(x))$ with $H \in \mathcal{H}F(x)$. Then
  \begin{align*}
    \dist(y, \xbar) &=  \dist({H}^{-}(\xbar, -F(x)), \xbar) \nonumber \\
                &=  \dist({H}^{-}(\xbar, -F(x)), {H}^{-}(\xbar, 0)) \nonumber \\
                &=  \dist({H}^{-}(\xbar, -F(x)), {H}^{-}(\xbar, -F(\xbar))).
  \end{align*}
  Here we use the property from~\eqref{eq:def-geodesical-subadditive}
  and then~\eqref{eq:newton-type-method-is-a-contraction-metric} to
  yield
  \begin{align*}
    \dist(y, \xbar) &=  \dist({H}^{-}(\xbar, -F(x)), {H}^{-}(\xbar, -F(\xbar))) \nonumber \\
                &\le \||{H}^{-}\||\|F(x) - F(\xbar) - H(x)(x, \xbar)\| \nonumber \\
                &\le (c + \varepsilon)\||{H}^{-}\||\dist(x, \xbar) \nonumber \\
                &\le (c + \varepsilon)\||{H}^{-}\||\dist(x, \xbar) \nonumber \\
                &\le (c + \varepsilon) \Omega \dist(x, \xbar),
  \end{align*}
  with $(c + \varepsilon) \Omega < 1$, showing the conclusion that the mapping is a
  quasi-contraction on $V$.
\end{proof}

Because of the Banach Fixed Point Theorem, quasi-contractivity is sufficient
to guarantee linear convergence of the iterates.
\begin{corollary}\label{cor:convergence of weak newton differentiable maps}
  Let $F: U \subseteq \bbM \to \bbRn$ be pointwise weakly Newton differentiable at $\xbar \in \bbM$ with
  $F(\xbar) = 0$.
  Denote the Newton differential of $F$ by $\mathcal{H}F$ and assume that $\forall x \in U$ all
  $H \in \mathcal{H}(x)$ are inversely compatible mappings, that is $\mathcal{H}(x) \subseteq GS_n(\bbM)$.
  Furthermore, assume that the set $\bigcup_{x \in U}\{\||{{H}^{-}}\||~|~H \in \mathcal{H}F(x)\}$
  is bounded by $\Omega > 0$ and let $c > 0$ be the limit
  in~\eqref{eq:def-newton-diff-weak-metric}, i.e.
  \begin{equation*}
    \lim_{x \to \xbar}\sup_{H \in \mathcal{H}F(x)}\frac{\|F(x) - F(\xbar) - H(x, \xbar)\|}{\dist(x, \xbar)} = c
  \end{equation*}
  be such that $c \Omega < 1$. Then there exists $V \subseteq \bbM$ such that the sequence $
  {\{x^k\}}_{k \in \bbN}$ generated
  by $\mathcal{N}_F$ is linearly convergent to $\xbar$ for all $x^0 \in V$.
\end{corollary}

Under the stronger assumption of Newton differentiability (as opposed to weak
Newton differentiability) the bound on $c \Omega$ can be removed. This is
expected, because under Newton differentiability, the constant $c$ is equal to 0
and as such the bound $c \Omega$ is satisfied for any $\Omega > 0$.
\begin{proposition}\label{prop:metric-newton-is-quasicontraction}
  Let $F: U \subseteq \bbM \to \bbRn$ be pointwise Newton differentiable at $\xbar$ with $F(\xbar) = 0$.
  Denote the Newton differential of $F$ by $\mathcal{H}F$ and assume that $\forall x \in U$ all
  $H \in \mathcal{H}(x)$ are inversely compatible mappings, that is $\mathcal{H}(x) \subseteq GS_n(\bbM)$.
  Assume that the set $\bigcup_{x \in U}\{\||{{H}^{-}}\||~|~H \in \mathcal{H}F(x)\}$
  is bounded by $\Omega > 0$. Then there exists a neighborhood $\xbar \in V \subseteq \bbM$ such that the
  mapping $\mathcal{N}_{\mathcal{H}F}$ is a quasi-contraction on $V$.
\end{proposition}
\begin{proof}
  From the definition of Newton differentiability there exists a neighborhood
  $\xbar \in V \subseteq U$ and a constant $c+\varepsilon$ with $(c+\varepsilon) \Omega < 1$ such that
  \begin{equation*}
    \forall x \in V, \quad, \sup_{H \in \mathcal{H}F(x)} \|F(x) - F(\xbar) - H(x)(x, \xbar)\| \le (c + \varepsilon)\dist(x, \xbar).
  \end{equation*}
  From this, the rest of the proof proceeds exactly as in the proof of
  Proposition~\ref{prop:weak-metric-newton-is-quasicontraction}.
\end{proof}

After obtaining convergence of the iterates to a fixed point, we can use
the rate provided by Newton differentiability to yield a convergence
rate for Newton-type methods.
\begin{theorem}[Superlinear Convergence of Newton-type Methods in Quasi-Metric
  Spaces]\label{thm:conv-newton-diff-metric}
  Let $F: U \subseteq \bbM \to \bbRn$ be pointwise Newton differentiable at $\xbar \in U$ with $F(\xbar) = 0$.
  Denote the Newton differential of $F$ by $\mathcal{H}F$ and assume that $\forall x \in U$ all
  $H \in \mathcal{H}F(x)$ are inversely compatible mappings, that is $\mathcal{H}(x) \subseteq GS_n(\bbM)$.
  Furthermore, assume that the set $\bigcup_{x \in U}\{\||{{H}^{-}}\||~|~H \in \mathcal{H}F(x)\}$
  is bounded by $\Omega > 0$. Then any sequence ${\{x^k\}}_{k \in \mathbb{N}}$ generated
  by~\eqref{eq:Metric-Newton-type-method} converges superlinearly to $\xbar$ for all
  $x^0$ near $\xbar$.
\end{theorem}
\begin{proof}
  Clearly, $\xbar$ is a fixed point of $\mathcal{N}_F$ because
  \begin{equation*}
    \forall H \in \mathcal{H}F(x)\quad {H}^{-}(\xbar, F(\xbar)) = {H}^{-}(\xbar, 0) = \xbar.
  \end{equation*}
  Using Proposition~\ref{prop:metric-newton-is-quasicontraction} we can conclude
  from Theorem~\ref{t:BFP} that $x^k$ converges to $\xbar$.

  For each $k \in \bbN$, we can use~\eqref{eq:def-geodesical-subadditive}, to show that
  \begin{align}
    \dist(x^{k+1}, \xbar) &= \dist({H(x^k)}^{-}(\xbar, F(x^k)), \xbar) \nonumber \\
                      &= \dist({H(x^k)}^{-}(\xbar, F(x^k)),
                        {H(x^k)}^{-}(\xbar, F(\xbar))) \nonumber \\
                      &\le \||{H(x)}^{-}\||\|F(x) - F(\xbar) - H(x)(x, \xbar)\|.
                        \label{e:random-name-metric}
  \end{align}

  We use the bound on $\||{H(x)}^{-}\||$, to show that
  \begin{equation*}
    \dist(x^{k+1}, \xbar) \le \Omega \|F(x^k) - F(\xbar) - H^k(x^k - \xbar)\|,
  \end{equation*}
  using~\eqref{eq:def-geodesical-subadditive} from the definition of inversely
  compatible pseudo-linear maps.
  Next, we need Newton differentiability to conclude that there exists a sequence,
  ${\{c_k\}}_{k \in \bbN}$, converging to $0$ such that
  \begin{equation}\label{e:random--metric-name3}
    \|F(x^k) - F(\xbar) - H^k(x^k, \xbar)\| \le c_k \dist(x^k, \xbar).
  \end{equation}
  Together,~\eqref{e:random-name-metric} and~\eqref{e:random--metric-name3} yield
  \begin{equation*}
    \dist(x^{k+1}, \xbar) \le \Omega c_k \dist(x^k, \xbar),
  \end{equation*}
  with $ \Omega c_k \to 0$ as $k \to \infty$, and hence the sequence convergences
  superlinearly. Since the sequence $\{x^k\}$ was arbitrary, this holds for all
  sequences, as claimed.
\end{proof}

\subsection{Higher Convergence Rates}
In the classic theory of Newton's method, quadratic convergence rates are
attained under stronger smoothness assumptions. In order for us to capture
this behavior, we have to expand the considered notion of Newton differentiability.
\begin{definition}[Rate of Newton Differentiability]
  Let $F: U \subseteq \bbM \to \bbRn$ be pointwise Newton differentiable at $\xbar$, with Newton
  differential $\mathcal{H}F:U \setto S_n(\bbM)$. The maximal number $\gamma \ge 1$ for which one has
  \begin{equation}\label{eq:def-newton-diff--rate}
    \lim_{x \to \xbar}\sup_{H \in \mathcal{H}F(x)}\quad
    \frac{\|F(x) - F(\xbar) - H(x, \xbar)\|}{{\dist(x, \xbar)}^\gamma} < \infty
  \end{equation}
  is called the \emph{rate} of Newton differentiability.
  Furthermore, if the limit in~\eqref{eq:def-newton-diff--rate} is $0$, $\gamma$ is
  called a \emph{super-rate}.
\end{definition}

These stronger versions of Newton differentiability will translate directly
into higher convergence rates of the Newton-type method associated to the
Newton differential.
\begin{theorem}[Faster Convergence]
  Let $F: U \subseteq \bbM \to \bbRn$ be pointwise Newton differentiable at $\xbar$ with $F(\xbar) = 0$
  with {\lparan}super-{\rparan}rate $\gamma > 1$.
  Denote the Newton differential of $F$ by $\mathcal{H}F$ and assume that $\forall x \in U$ all
  $H \in \mathcal{H}F(x)$ are inversely compatible mappings, that is $\mathcal{H}(x) \subseteq GS_n(\bbM)$.
  Furthermore, assume that the set $\bigcup_{x \in U}\{\||{{H}^{-}}\||~|~H \in \mathcal{H}F(x)\}$
  is bounded by $\Omega$. Then any sequence ${\{x^k\}}_{k \in \mathbb{N}}$ generated
  by~\eqref{eq:Metric-Newton-type-method} converges with
  {\lparan}super-{\rparan}rate $\gamma$ to $\xbar$ for all $x^0$
  near $\xbar$.
\end{theorem}
\begin{proof}
  Following the exact line of reasoning as in the proof of
  Theorem~\ref{thm:conv-newton-diff-metric}, we can use the rate $\gamma$
  in~\eqref{e:random-name-metric}, concluding that there exists a convergent
  sequence ${\{c_k\}}_{k \in \bbN}$ such that
  \begin{equation*}
    \dist(x^{k + 1}, \xbar) \le c_k{\dist(x^k, \xbar)}^{\gamma}.
  \end{equation*}
  Furthermore, if $\gamma$ is a super rate, we know that $\lim_{k \to \infty}c_k = 0$,
  proving the desired conclusion.
\end{proof}

\section{The Kantorovich-type Theorem}
In order to obtain convergence rates for Newton-type theorems, the existence of
a zero must be assumed. In this section, we aim to establish sufficient
conditions for the existence of such a point, thus in effect providing an
inverse function theorem. To this effect, a vast body of work has been
undertaken in the past, culminating with Smale's $\alpha$-Theory~\cite{Sma86Newt} and
the Newton-Kantorovich Theorem. The former provides a test for the solvability
of the equation by using information about all the function's derivatives at a
given point and as such it is not suitable for our nonsmooth and nonlinear
context. The original Newton-Kantorovich Theorem was first proven by Leonid
Kantorovich in 1948 in~\cite{Kan48Func} and guarantees the existence of a zero
of a Lipschitz smooth function by requiring only information about the function
and its Jacobian at a given point. This theorem can be easily adapted to our
Newton differentiability context by strengthening the conditions satisfied by
the values of the function and its Newton differential at a point.

\begin{definition}\label{def:strong-invertably-compatible}
  An inversely compatible pseudo-linear mapping $H:\bbM \times \bbM \to \bbRn$ is called
  \emph{strongly inversely compatible} if
  \begin{equation*}
    \forall x, \forall v \in \bbRn,\quad v = H(x, H^{-}(x, v)),
  \end{equation*}
  and
  \begin{equation}
    \forall x, \forall v \in \bbRn,\quad x =H^{-}(x, v) \Rightarrow v = 0.
  \end{equation}
  The set of all such mappings is denoted by $SGS_n(\bbM)$.
\end{definition}
\begin{remark}
  This stronger invertibility condition is not required in order to obtain
  convergence rate results, but it is required for the Kantorovich-type theorem.
\end{remark}
\begin{remark}
  In the context of Example~\ref{ex:metric-space-euclidean}, an invertible
  linear map $T$ induces a strong-inversely compatible map.
\end{remark}

\begin{definition}\label{def:h-type-smooth}
  A mapping $H:\bbM \to GS_n(\bbM)$ is called \emph{pointwise h-smooth at $x^0$} if
  there exist $\kappa > 0$ and $\alpha > 0$ such that
  \begin{equation*}
    \forall x \in \bbM, \forall y, z \in \bbM\quad \dist({H(x)}^{-}(x, H(x^0)(z, y)), x)
    \le (1 + \kappa{\dist(x, x^0)}^{\alpha})\dist(y, z).
  \end{equation*}
  A mapping $\mathcal{H}:\bbM \setto GS_n(\bbM)$ has a \emph{pointwise h-smooth selection at $x^0$} if
  there exists a pointwise h-smooth at $x^0$ mapping $H:\bbM \to GS_n(\bbM)$ with
  $H(x) \in \mathcal{H}(x)$ for all $x \in \bbM$.
\end{definition}
\begin{remark}
  In the setting of Example~\ref{ex:metric-space-euclidean},
  $H:\bbRn \to GS_n(\bbRn)$ induced
  by $T:\bbRn \to \bbRnxn$ is h-smooth at $x^0$ if $T$ is pointwise Hölder
  continuous at $x^0$, with constant $\alpha > 0$, and $T^{-1}$ is uniformly bounded,
  i.e $\exists M < \infty$ such that
  for all $x$, $\|{T(x)}^{-1}\| \le M$. To see
  this, we compute
  \begin{align} \label{eq:in-remark-h-smooth}
    \dist(x, {H(x)}^{-}(x, H(x^0)(z, y)))
    &= \|x - {H(x)}^{-}(x, H(x^0)(z, y))\| \nonumber \\
    &= \|x - x - {T(x)}^{-1}H(x^0)(z, y)\| \nonumber \\
    &= \|{T(x)}^{-1}T(x^0)(y - z)\| \nonumber \\
    &\le \|{T(x)}^{-1}T(x^0)\|\|y - z\| \nonumber \\
    &= \|{T(x)}^{-1}(T(x^0) - T(x)) + \Id\|\|y - z\| \nonumber \\
    &\le (1 + \|{T(x)}^{-1}\|\|T(x^0) - T(x)\|)\|y - z\|.
  \end{align}
  Using the pointwise Hölder continuity of $T$, we know that
  \begin{equation*}
    \|T(x^0) - T(x)\| \le \kappa\|x - x^0\|^\alpha
  \end{equation*}
  and from~\eqref{eq:in-remark-h-smooth} we derive the conclusion.

  This example justifies the name h-smooth.
\end{remark}

The proof of the next theorem uses what Kantorovich called the general theory
of approximate methods, which solves problems by first constructing
an easier to solve instance of the problem and then relating this to the
original instance. The key idea of the proof developed here is that
of analyzing a Newton-type method applied to the auxiliary function defined by
equations~\eqref{eq:kantorovich-key-property-metric}%
~\eqref{eq:kantotovich-ft-ge-zero-mertric}%
~\eqref{eq:kantorovich-assumption-f2-metric}
and~\eqref{eq:kantorovich-assumption-f3-metric}.
This idea follows the lines presented by Ortega in~\cite{Ort68Then}.
\begin{theorem}[Kantorovich-type Theorem on Metric Spaces]\label{thm:kant}
  Let $\bbM$ be a quasi-metric space and $F: U \subseteq \bbM \to \bbRn$ be uniformly Newton
  differentiable on $U$ with rate $\gamma$ and the Newton differential
  $\mathcal{H}F:U \setto SGS_n(\bbM)$ having a h-smooth selection
  (see Definition~\ref{def:h-type-smooth})
  denoted by $H_F$. Concretely, let $L>0 $ and $\gamma \in [1, 2]$ be such that
  \begin{equation*}
    \forall x \in \bbM, \forall y, z \in \bbM\quad \dist(x, {H_F(x)}^{-}(x, H_F(x^0)(z, y)))
    \le (1 + L{\dist(x, x^0)}^{\gamma-1})\dist(y, z)
  \end{equation*}
  and such that
  \begin{equation*}
    \forall x, y \in U,\quad  \|F(x) - F(y) - H_F(x)(x, y)\| \le \frac{L}{2}{\dist(x, y)}^{\gamma}.
  \end{equation*}

  Let $x^0 \in U$ and assume there exists $B < \infty$ such that
  \begin{equation*}
    B = \sup_{x \ne y \in \bbM}\frac{\dist(x, y)}{\|H_F(x^0)(x, y)\|}
  \end{equation*}
  and set $ \eta:=\dist(x^0, {H_F(x^0)}^{-}(x^0, F(x^0)))$.
  Suppose further that there exists a constant $M^* > 0$ such that for all
  $x \in U$
  \begin{equation*}
    \||{H_F(x)}^{-}\|| \le M^*.
  \end{equation*}
  Assume that for all $x \in U$,
  \begin{equation*}
    {H_F(x)}^{-}(x, -F(x)) \in U.
  \end{equation*}

  Furthermore assume there exists $\tbar \in (0, \infty)$ and a function
  $f \in \mathcal{C}^2[0, \tbar]$ with Lipschitz continuous second derivative such that for all
  $t < \tbar$
  \begin{enumerate}[(a)]
  \item
    \begin{equation}\label{eq:kantorovich-key-property-metric}
      \frac{LB}{2}{\left(-\frac{f(t)}{f'(t)}\right)}^\gamma \le
      f\left(t - \frac{f(t)}{f'(t)}\right),
    \end{equation}
  \item
    \begin{equation}\label{eq:kantotovich-ft-ge-zero-mertric}
      f(0) = \eta,\quad
      f(t) > 0,\quad     f(\tbar) = 0,
    \end{equation}
  \item
    \begin{equation}\label{eq:kantorovich-assumption-f2-metric}
      f'(t) < 0,\quad f'(t) \ge -{(1 + L t^{\gamma-1})}^{-1},
    \end{equation}
  \item
    \begin{equation}\label{eq:kantorovich-assumption-f3-metric}
      f''(t) > 0.
    \end{equation}
  \end{enumerate}
  Then, for any $x^0$, the sequence ${\{x^k\}}_{k \in \mathbb{N}}$ is Cauchy and
  $\lim_{k \to \infty} F(x^k) = 0$.
\end{theorem}

\begin{proof}
  The first step of the proof consists in developing a convergence result
  for Newton's method applied to solving $f(t) = 0$. For this we consider
  the function $N:[0, \tbar] \to [0, \tbar]$ defined by
  \begin{equation*}
    N(t) = t - \frac{f(t)}{f'(t)}.
  \end{equation*}

  The case $\eta = 0$ in~\eqref{eq:kantotovich-ft-ge-zero-mertric} means that
  $F(x^0) = 0$, and we are done.
  Otherwise~\eqref{eq:kantotovich-ft-ge-zero-mertric} shows that $\tbar$ is a fixed
  point of $N$ i.e. $N(\tbar) = \tbar$. Furthermore, because $f(t) > 0$ for all
  $t \in (0, \tbar)$ we can deduce that $\tbar$ is the unique fixed point.

  Next, we have to analyze the behavior of the fixed point iteration of the
  mapping $N$. From the definition of $f$, we can easily
  conclude that for all $t \in (0, \tbar)$
  \begin{equation}\label{e:Thing1}
    -\frac{f(t)}{f'(t)} \ge 0,
  \end{equation}
  so $n(t) \ge t$.

  Using Taylor's Theorem, we can expand $f$ at the unique root $\tbar$
  \begin{equation*}
    0 = f(\tbar) = f(t) + f'(t)(\tbar - t) + \frac{1}{2}f''(\xi){(\tbar - t)}^2.
  \end{equation*}
  Simplifying and rearranging gives
  \begin{equation}\label{e:Thing2}
    \tbar - N(t) = -\frac{f''(\xi)}{2 f'(t)}{(\tbar - t)}^2 \ge 0,
  \end{equation}
  showing that $N(t) \le \tbar$.

  Equations~\eqref{e:Thing1} and~\eqref{e:Thing2} together show that for any
  $t_0$ the sequence ${(t_k)}_{k \in \bbN}$
  generated by
  \begin{equation}\label{eq:proof kantorovich def tk}
    t_{k+1} = N(t_k)
  \end{equation}
  is monotonically increasing and bounded, so
  convergent. Let $t^*$ denote the limit of this sequence and as such we can
  use the fact that $N(t^*) = t^*$ to conclude that $f(t^*) = 0$. Because
  of~\eqref{eq:kantotovich-ft-ge-zero-mertric}, we know that $\tbar$ is the unique zero
  of $f$ on $[0, \tbar]$, yielding $\tbar = t^*$. This constitutes the part of the proof
  in which a different instance of the problem is constructed and studied,
  as per the general outline of an approximate method described by Kantorovich.

  For the remainder of the proof we consider the sets
  \begin{equation*}
    \Sigma(t) = \{x \in U ~|~\|x - x^0\| \le t, \dist(x, {H_F(x^0)}^{-}(x, F(x))) \le f(t)\},
  \end{equation*}
  for any $t \in (0, \infty)$.
  These sets will help us relate the behavior of the Newton-type methods
  for $f$ and $F$.

  Let ${(x^k)}_{k \in \bbN}$ be generated by iterating $\mathcal{N}_F$, where
  \begin{equation*}
    x^{k + 1} = {H_F(x^k)}^{-}(x^k, -F(x^k)).
  \end{equation*}

  Now consider ${(t_k)}_{k \in \bbN}$ generated by $t_{k+1} = N(t_k)$ with $t_0 = 0$.
  We will show using induction that $x_{k} \in  \Sigma(t_k)$. The first step consists
  in concluding that $x^0 \in \Sigma(0)$, i.e.
  \begin{equation*}
    \eta = \dist(x^0, {H_F(x^0)}^{-}(x^0, F(x^0))) \le f(0)
  \end{equation*}
  and this is exactly the condition from~\eqref{eq:kantotovich-ft-ge-zero-mertric}.

  We proceed by assuming that $x^{k} \in  \Sigma(t_k)$ and showing that
  $x^{k+1} \in  \Sigma(t_{k+1})$. Indeed, using the \emph{strong inverse compatibility}
  of $H(x^0)$ from Definition~\ref{def:strong-invertably-compatible}, we can derive
  \begin{align}\label{eq:proof kant what is this step}
    \dist(x^{k+1}, x^k)
    &= \dist({H(x^{k})}^{-}(x^{k}, -F(x^k)), x^k) \nonumber \\
    &= \dist({H(x^{k})}^{-}(x^{k},  H(x^0)(x^k, {H(x^0)}^{-}(x^k, -F(x^k)))), x^k).
  \end{align}

  Using the \emph{pointwise h-smoothness} of $H$ at $x^0$ with $x^k$ and
  ${H(x^0)}^{-}(x^k, -F(x^k))$ as $y$ and $z$ respectively,
  in~\eqref{eq:proof kant what is this step} we obtain the bound
  \begin{equation*}
    \dist(x^{k+1}, x^k) \le (1 + L{\dist(x^k, x^0)}^{\gamma-1})\dist(x^k, {H(x^0)}^{-}(x^k, -F(x^k))).
  \end{equation*}
  From the definition of $\Sigma$ and the fact that $x^k \in \Sigma(t_k)$ we know that
  \begin{equation*}
    (1 + L{\dist(x^k, x^0)}^{\gamma-1})\dist(x^k, {H(x^0)}^{-}(x^k, -F(x^k)))
    \le (1 + L{t_k}^{\gamma-1})f(t_k),
  \end{equation*}
  and from~\eqref{eq:kantorovich-assumption-f2-metric}
  we can conclude the key bound
  \begin{equation}\label{eq:proof kant one last step in proof is this}
    \dist(x^{k+1}, x^k) \le -\frac{f(t_k)}{f'(t_k)}.
  \end{equation}

  In the next part of the proof, we use \emph{strong-inverse compatibility}
  of $H(x^k)$ to conclude that
  \begin{equation}\label{eq:proof kantorovich from strong invertably compatibililty}
    \|F(x^{k+1})\| = \|F(x^{k}) - F(x^{k+1})
    - {H(x^{k})}(x^k, {H(x^k)}^{-}(x^k, -F(x^k)))\|.
  \end{equation}
  Using Newton differentiability
  in~\eqref{eq:proof kantorovich from strong invertably compatibililty}, we can
  bound
  \begin{equation}\label{eq:proof kantorovich bound with gamma}
    \|F(x^{k+1})\| \le \frac{L}{2}{\dist(x^{k+1}, x^{k})}^\gamma.
  \end{equation}
  Combining~\eqref{eq:proof kantorovich bound with gamma}
  with~\eqref{eq:proof kant one last step in proof is this} and
  using~\eqref{eq:kantorovich-key-property-metric} gives
  \begin{equation}\label{eq:kant proof last part useful}
    \|F(x^{k+1})\| \le \frac{L}{2}{\left(-\frac{f(t_k)}{f'(t_k)}\right)}^\gamma \le \frac{f(t_{k+1})}{B}.
  \end{equation}
  The definition of $B$ then proves that
  \begin{equation}\label{eq:--new-- eq1}
    \dist(x^{k+1}, {H_F(x^0)}^{-}(x^{k+1}, F(x^{k+1})))
    \le B\|H_F(x^0)(x^{k + 1}, {H_F(x^0)}^{-}(x^{k+1}, F(x^{k+1})))\|,
  \end{equation}
  while the strong inverse compatibility shows
  \begin{equation}\label{eq:--new-- eq2}
    \|H_F(x^0)(x^{k + 1}, {H_F(x^0)}^{-}(x^{k+1}, F(x^{k+1})))\| = \|F(x^{k+1})\|.
  \end{equation}
  Combining~\eqref{eq:kant proof last part useful} with~\eqref{eq:--new-- eq1}
  and~\eqref{eq:--new-- eq2} gives
  \begin{equation*}
    \dist(x^{k+1}, {H_F(x^0)}^{-}(x^{k+1}, F(x^{k+1}))) \le f(t_{k+1}),
  \end{equation*}
  and this is one of the two requirements for $x^{k + 1} \in \Sigma(t_{k+1})$.

  Clearly, substituting in~\eqref{eq:proof kant one last step in proof is this}
  the definition of $t_{k+1}$ from~\eqref{eq:proof kantorovich def tk}
  \begin{equation}\label{eq:proof kant one last step in proof is this v2}
    \dist(x^{k+1}, x^k) \le t_{k+1} - t_k.
  \end{equation}

  It remains to show that $\dist(x^{k+1}, x^0) \le t_{k+1}$. Indeed, using the
  triangle inequality
  \begin{equation*}
    \dist(x^{k+1}, x^0) \le \dist(x^{k+1}, x^{k}) + \dist(x^k, x^0)
  \end{equation*}
  and from~\eqref{eq:proof kant one last step in proof is this v2} and the
  induction hypothesis
  \begin{equation*}
    \dist(x^{k+1}, x^0) \le t_{k+1} - t_k + t_k,
  \end{equation*}
  completing the induction part of the proof and shows that for all
  $k \in \bbN$, $x^{k} \in \Sigma(t_k)$.

  The last part of the proof consists in looking at the convergence of the
  sequence ${\{x^k\}}_{k \in \bbN}$. A simple telescoping argument shows that
  \begin{align*}
    \dist(x^{m}, x^{n}) &\le \dist(x^{m}, x^{m-1}) + \cdots + \dist(x^{n+1}, x^n) \nonumber \\
    &\le |t_m - t_{m-1}| + \cdots + |t_{n+1} - t_n| \nonumber \\
    &\le t_m - t_{m-1} + \cdots + t_{n+1} - t_n = t_m - t_n.
  \end{align*}
  Because the sequence ${\{t_k\}}_{k \in \bbN}$ is Cauchy we deduce that
  ${\{x^{k}\}}_{k \in \bbN}$ is Cauchy.

  In order to complete the proof, we take the limit as $k \to \infty$
  in~\eqref{eq:kant proof last part useful}, using the fact that $f$ is continuous
  and $f(\lim_{k \to \infty}t_k) = 0$ to conclude that
  \begin{equation*}
    \lim_{k \to \infty}\|F(x^k)\| \le \lim_{k \to \infty}\frac{f(t_k)}{\eta} = 0.
  \end{equation*}
\end{proof}

\begin{remark}
  As opposed to the convergence rate proof, the Kantorovich-type theorem
  generates a Cauchy sequence, and we need to further impose completeness on the
  quasi-metric space in order to obtain a limit point. Furthermore, we also
  need to assume continuity of $F$ at this limit point in order to be able
  to guarantee that $F(\lim_{k \to \infty}x^k) = 0$.
\end{remark}

\section{A Numerical Example}
In this section we will investigate the behavior of our Newton-type algorithm
for a simple toy optimization problem defined on a cubic complex.
For this purpose we
first consider a finite binary tree $B$ with root $r$ and with its usual
distance $\dist_B$, defined as the minimal number of edges of a path between
two nodes.
Such a tree is a uniquely geodesic space, so between any two points $b_x$
and $b_y$ there exists a unique path $b_0, b_1, \dots, b_N$ with $b_0= b_x$ and
$b_N = b_y$. This path allows use to introduce the auxiliary functions
$\gamma_{b_x, b_y}:[0, 1] \to [0, 1]$,
\begin{equation*}
  \gamma_{b_x, b_y}(x) = \left \{\begin{array}{ll}
    x & \text{ if } b_1 \text{ is a direct descendant of } b_0 \\
    1 - x & \text{ if } b_1 \text{ is the parent of } b_0.
  \end{array} \right .
\end{equation*}

We define the disjoint union $\tilde{\bbM} = \dot{\bigcap}_{b \in B}[0, 1]$ with
the semimetric
\begin{equation*}
  \dist((b_x, x), (b_y, y)) = \left \{\begin{array}{ll}
    |x - y| &\text{ if } \dist_B(b_x, b_y) = 0 \\
    \gamma_{b_x, b_y}(x) + \gamma_{b_y, b_x}(y) + \dist_B(b_x, b_y) - 1 & \text{ else. }
  \end{array} \right .
\end{equation*}
To construct the metric space, we simply identify points with distance $0$
between them, thus setting $\bbM = \tilde{\bbM}/\dist$. It is clear that this space is
compact in the topology induced by the distance as it is essentially a closed
subset of the product of finitely many compact spaces.
\begin{remark}
  It is important to note that $\bbM$ is a metric space, and not just a
  quasi-metric space.
\end{remark}
\begin{remark}
  The effect of this identification is to merge points $(b_x, 1)$ with points
  $(b_y, 0)$, if $b_y$ is a direct descendant of $b_x$.
\end{remark}
The space constructed has the structure of a cubical complex (see~\cite{BriHae99Mka}),
thus it mixes the combinatorial structure of a binary tree with the Euclidean
structure of $[0, 1]$. Figure~\ref{fig:space} is helpful in visualizing this
abstract construction.
\begin{figure}
  \centering
  \begin{tikzpicture}
    \node {$0$}
    child { node {$r$}
      child { node {$b_x$}
        child { node {$b_1$} edge from parent
          child { node {} edge from parent }
          child { node {$b_2$} edge from parent
            child { node {} edge from parent[draw=none] }
            child {
              node {$b_y$} edge from parent
              node[right] {$.75$}
            }
          }
        }
        child { node {} edge from parent}
        edge from parent node[left] {$.6$} }
      child { node {} edge from parent
        child { node {} edge from parent[draw=none] }
        child { node {} edge from parent }
      }
    };
    \node at (-0, -0.2) [draw, circle, fill=red, scale=.3] {};
    \node at (-0.4, -2.31) [draw, circle, fill=blue, scale=.3] {};
    \node at (-0.3, -6.9) [draw, circle, fill=orange, scale=.3] {};
  \end{tikzpicture}
  \caption{An example of the metric space $\bbM$ constructed using a binary
    tree. The points ${\color{red}(r, 0)}$, ${\color{blue}(b_x, .6)}$ and
    ${\color{orange}(b_y, .75)}$ are marked, together with
    the shortest path $b_x, b_1, b_2, b_y$ between $b_x$ and $b_y$}\label{fig:space}
\end{figure}
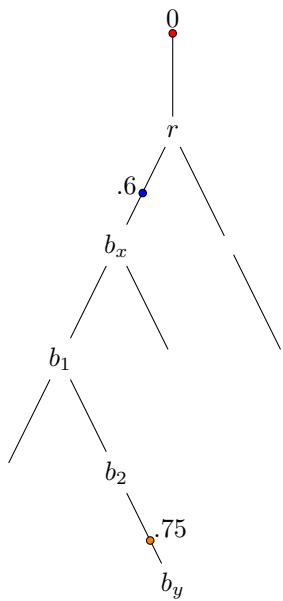
\begin{example}
  In order to compute the distance between the points $(b_x, .6) \in \bbM$ and
  $(b_y, .75) \in \bbM$ from the space $\bbM$ represented in Figure~\ref{fig:space}, we
  need to consider the path $b_x, b_1, b_2, b_y$. Since $b_1$ is a direct
  descendant of $b_x$, we compute $\gamma_{b_x, b_y}(.6) = .4$ and because $b_2$ is
  the parent of $b_y$, we compute $\gamma_{b_y, b_x}(.75) = .75$. As such
  \begin{equation*}
    \dist((b_x, .6), (b_y, .75)) = .4 + .75 + 3 - 1 = 3.15.
  \end{equation*}
\end{example}
\begin{example}
  For another clarifying computation, we can look at $\dist((r, 0), (b_x, .6))$.
  The second node in the shortest path between $r$ and $b_x$ is clearly $b_x$
  and a direct descendant of $r$, so $\gamma_{r, b_x}(0) = 1$. For the same reason,
  $\gamma_{b_x, r}(.6) = .6$. As such
  \begin{equation*}
    \dist((r, 0), (b_x, .6)) = 1 + .6 + 1 - 1 = 1.6.
  \end{equation*}
\end{example}

For the objective function, we introduce
$f:[0, \infty) \to \bbR$ a strongly convex $\mathcal{C}^\infty$ function and assume there exists a unique
minimizer of $f$ in the open interval
$(0, \max_{(b_x, x) \in \bbM} \dist((r, 0), (b_x, x)))$. It is clear that
$f''$ is not $0$ at any point in the domain of $f$. The objective
function $F:\bbM \to \bbR$ is then defined as $F(b_x, x) = f'(\dist((r, 0), (b_x, x)))$.
Solving the equation $F(b_{\xbar}, \xbar) = 0$ is then equivalent to finding a
point $(b_{\xbar}, \xbar)$ such that $f(\dist((r, 0), (b_x, x))) = \min_{x \in [0, \infty)}f(x)$,
and thus solving the minimization problem
\begin{equation}\label{eq:num example}
  \argmin_{(b_x, x) \in \bbM} f(\dist((r, 0), (b_x, x))).
\end{equation}
\begin{remark}
  The solution to~\eqref{eq:num example} always exists, but it
  is not necessary unique, in spite of
  the uniqueness of the minimizer of $f$. In fact, averaging over all possible
  trees and denoting the minimizer of $f$ by $\tbar$, the number of solutions
  to~\eqref{eq:num example} is $\mathcal{O}(\log \tbar)$.
\end{remark}

For simplicity, let us denote by $d_r:\bbM \to \bbR$ the map
$(b_x, x) \mapsto \dist((r, 0), (b_x, x))$.
Using the reverse triangle inequality, it is easy to see that this map
is Lipschitz continuous with constant $1$.

The next step in implementing the Newton-type method consists in constructing
a Newton differential. Pursuant to this, consider the single valued map
$\mathcal{H}F:\bbM \to S_1(\bbM)$
\begin{equation*}
  \mathcal{H}F(b_x, x)((b_y, y), (b_z, z)) =
  f''(d_r(b_x, x))(d_r(b_z, z) - d_r(b_y, y))
\end{equation*}
Let $(b_{\xbar}, \xbar)$ be a solution of~\eqref{eq:num example} and Taylor expand $f'$
around $d_r(b_{\xbar}, \xbar)$, concluding that there exists $h:\bbR \to \bbR$ such
that $\lim_{t \to d_r(b_{\xbar}, \xbar)}h(t) = 0$ and
\begin{equation*}
  f'(t) = f'(d_r(b_{\xbar}, \xbar)) + f''(d_r(b_{\xbar}, \xbar))
  (t - d_r(b_{\xbar}, \xbar)) + h(t)|t - \dist((r, 0), (b_{\xbar}, \xbar)|.
\end{equation*}
We check the
Newton differentiability of $F$ at $(b_{\xbar}, \xbar)$ by computing
\begin{align*}
  &\quad\frac{|F(b_x, x) - F(b_{\xbar}, \xbar) - \mathcal{H}F(b_x, x)((b_x, x), (b_{\xbar}, \xbar))|}{\dist(
  (b_x, x), (b_{\xbar}, \xbar))} \\
  &= \frac{|f'(d_r(b_x, x)) - f'(d_r(b_{\xbar}, \xbar)) - f''(d_r(b_x, x))(d_r(b_x, x) - d_r(b_{\xbar}, \xbar))
    |}{\dist((b_x, x), (b_{\xbar}, \xbar))} \\
  &\le \frac{|h(d_r(b_x, x))||d_r(b_x, x) - d_r(b_{\xbar}, \xbar)| + |f''(d_r(b_x, x)) - f''(d_r(b_{\xbar}, \xbar))||d_r(b_x, x) - d_r(b_{\xbar}, \xbar)|
    |}{\dist((b_x, x), (b_{\xbar}, \xbar))} \\
  &\le \frac{|h(d_r(b_x, x))||\dist((b_x, x), (b_{\xbar}, \xbar))| + |f''(d_r(b_x, x)) - f''(d_r(b_{\xbar}, \xbar))||\dist((b_x, x), (b_{\xbar}, \xbar))|
    |}{\dist((b_x, x), (b_{\xbar}, \xbar))} \\
  &\le |h(d_r(b_x, x))| + |f''(d_r(b_x, x)) - f''(d_r(b_{\xbar}, \xbar))|, \\
\end{align*}
and taking the limit as $(b_x, x) \to (b_{\xbar}, \xbar)$ shows
\begin{align*}
  \lim_{(b_x, x) \to (b_{\xbar}, \xbar)}&\frac{|F(b_x, x) - F(b_{\xbar}, \xbar) - \mathcal{H}F(b_x, x)((b_x, x), (b_{\xbar}, \xbar))|}{\dist(
  (b_x, x), (b_{\xbar}, \xbar))} \\
  &\le \lim_{(b_x, x) \to (b_{\xbar}, \xbar)} |h(d_r(b_x, x))| + |f''(d_r(b_x, x)) - f''(d_r(b_{\xbar}, \xbar))|, \\
  &= 0.
\end{align*}
This completes the proof that $F$ is pointwise Newton differentiable at its
roots.

In order to implement a Newton-type method one has to provide a quasi-inverse
map for $\mathcal{H}F$. For this purpose, we have to consider a point $(m, 1)$ such
that $d_r(m, 1) = \max_{(b_x, x) \in \bbM}d_r(b_x, x)$. Such a point necessary
exists because of the compactness of $\bbM$. Next, we consider the path
$\pi:[0, d_r(m, 1)] \to \bbM$ defined by
\begin{equation*}
  \pi(t) = (b_{\lfloor t \rfloor}, t - \lfloor t \rfloor),
\end{equation*}
where $\lfloor t \rfloor$ is the integer part of $t$ and $b_0, b_1,\dots,b_{\lfloor d_r(m, 1) \rfloor}$ is a
path between $r$ and $m$ in the binary tree, such that $b_0 = r$ and
$b_{\lfloor d_r(m, 1) \rfloor } = m$.

Using this map, we can define the quasi-inverse of the Newton differential as
\begin{equation*}
  \mathcal{H}F^{-}(b_x, x)((b_y, y), v)
  = \left \{\begin{array}{ll}
    \pi(d_r(b_y, y) + f''(d_r(b_x, x))^{-1}v)
    & \text{ if } d_r(b_y, y) + f''(d_r(b_x, x))^{-1}v \in [0, d_r(m, 1)] \\
    (r, 0)
    &\text{ if } d_r(b_y, y) + f''(d_r(b_x, x))^{-1}v < 0 \\
    (m, 1)
    &\text{ if } d_r(b_y, y) + f''(d_r(b_x, x))^{-1}v > d_r(m, 1).
  \end{array} \right .
\end{equation*}

For convenience, let us introduce the function $\clamp:\bbR \to [0, d_r(m, 1)]$
defined by $\clamp(t) = \min \{ \max \{ t, 0\}, d_r(m, 1)\}$. It is clear that
this map is Lipschitz continuous with constant $1$, and that
\begin{equation*}
  \mathcal{H}F^{-}(b_x, x)((b_y, y), v)
  = \pi(\clamp(d_r(b_y, y) + f''(d_r(b_x, x))^{-1}v))).
\end{equation*}
It also follows from the definition that
\begin{equation}\label{eq:proof psinv 1}
  d_r(\mathcal{H}F^{-}(b_x, x)((b_y, y), v)) = \clamp(d_r(b_y, y) + f''(d_r(b_x, x))^{-1}v),
\end{equation}
and that
\begin{equation}\label{eq:proof psinv 2}
  \dist(\mathcal{H}F^{-}(b_x, x)((b_y, y), v), \mathcal{H}F^{-}(b_x, x)((b_z, z), w))
  = |d_r(\mathcal{H}F^{-}(b_x, x)((b_z, z), w)) - d_r(\mathcal{H}F^{-}(b_x, x)((b_y, y), v))|,
\end{equation}
since all the point in the image of $\mathcal{H}F^{-}(b_x, x)$ belong to the same
path from the root of the tree to $m$.

Pursuant to proving that the map $\mathcal{H}F^{-}(b_x, x)$ is indeed the quasi-inverse
of $\mathcal{H}F(b_x, x)$, we combine~\eqref{eq:proof psinv 1}
with~\eqref{eq:proof psinv 2} and the Lipschitz continuity of $\clamp$ to
obtain
\begin{align*}
  \dist(\mathcal{H}F^{-}(b_x, x)
  &((b_y, y), v), \mathcal{H}F^{-}(b_x, x)((b_y, y), v)) \\
  &= |d_r(\mathcal{H}F^{-}(b_x, x)((b_z, z), w)) - d_r(\mathcal{H}F^{-}(b_x, x)((b_y, y), v))| \\
  &= |\clamp(d_r(b_z, z) + f''(d_r(b_x, x))^{-1}w) - \clamp(d_r(b_y, y) + f''(d_r(b_x, x))^{-1}v)| \\
  &\le |(d_r(b_z, z) + f''(d_r(b_x, x))^{-1}w) - (d_r(b_y, y) + f''(d_r(b_x, x))^{-1}v)| \\
  &= |f''(d_r(b_x, x))^{-1}||-v + w + f''(d_r(b_x, x))(d_r(b_z, z) - d_r(b_y, y))| \\
  &= |f''(d_r(b_x, x))^{-1}||-v + w + \mathcal{H}F(b_x, x)((b_y, y), (b_z, z))|,
\end{align*}
completing the proof.
\begin{remark}
  The need to pick one distinguished path, $\pi$, in order to construct the
  quasi-inverse is indicative the necessity for rather ad hoc
  constructions in the field of optimization on metric spaces.
\end{remark}

From this calculation, it is clear that $\||\mathcal{H}F^{-}(b_x, x)\|| \le
|f''(d_r(b_x, x))^{-1}|$ and since $f \in \mathcal{C}^\infty$ and $\bbM$ is compact, we can
conclude that $\||\mathcal{H}F^{-}(b_x, x)\||$ is bounded. As such, all the assumptions
of Theorem~\ref{thm:conv-newton-diff-metric} are satisfied. This proves that, for any
$(b_{x_0}, x_0)$ close enough to a $(b_{\xbar}, \xbar)$, the iteration
\begin{equation*}
  (b_{x_{k + 1}}, x_{k + 1})) =
  \pi(\clamp(d_r(b_{x_k}, x_k) - f''(d_r(b_{x_k}, x_k))^{-1}f'(d_r(b_{x_k}, x_k)))),
\end{equation*}
converges superlinearly to $(b_{\xbar}, \xbar)$.

\section{Conclusions and Limitations}
The results of this work are formulated in quite general settings and are
intended primarily to serve as a stepping stone to a more practical optimization
theory in metric spaces. The necessary algebraic constructs for casting an
optimization problem as a root-finding problem are still under active research
and will be the subject of a following article. While the general metric theory
is still under development, the results of this article, in particular the
calculus of Newton differentiability and the Kantorovich theorem, can be applied
to nonsmooth problems on nonsmooth subsets of Euclidean spaces. In this context,
for a practical application of the Kantorovich result, numerical simulation for
the existence of the function $f$ can be implemented. Compared to actually
running a Newton-type method, this numerical certificate of existence for a
root has the advantage that its computational complexity does not depend on the
dimension of the ambient space.

\section{Acknowledgments}
An earlier version of this work was part of the author's PhD thesis, completed at
the University of G\"ottingen under the supervision of D. Russell Luke. The author
would like to thank D. Russell Luke for his guidance and advice during the
undertaking of the PhD. The author would also like to thank his
postdoc mentor, Sorin-Mihai Grad. This work has been partially founded
by ANR-22-EXES-0013.

{\printbibliography}

\end{document}